\documentclass[11pt,reqno]{amsart}
\usepackage{amsmath, amsfonts, amsthm, amssymb, color, cite, soul}
\usepackage{mathrsfs}
\usepackage{dsfont}
\usepackage{upref}
\usepackage{indentfirst}
\usepackage{appendix}
\usepackage{enumerate}
\usepackage[bookmarks,colorlinks,citecolor=blue,linkcolor=red]{hyperref}
\usepackage{bm}
\usepackage{graphicx}
\allowdisplaybreaks
\numberwithin{equation}{section}
\numberwithin{figure}{section}
\newtheorem{theorem}{Theorem}[section]

\newtheorem{remark}{Remark}[section]
\newtheorem{lemma}{Lemma}[section]

\def\eps{\epsilon}

\author[Y.-Z. Chen]{Yazhou Chen}
\address{College of Mathematics and Physics, Beijing University of
Chemical Technology, Beijing 100029, China}
\email{chenyz@mail.buct.edu.cn}

\author[Q.-L. He]{Qiaolin He}
\address{College of Mathematics, Sichuan University, Sichuan 610065, China}
\email{qlhejenny@scu.edu.cn}

\author[X.-D. Shi]{Xiaoding Shi}
\address{College of Mathematics and Physics, Beijing University of
Chemical Technology, Beijing 100029, China}
\email{shixd@mail.buct.edu.cn}

\author[X.-P. Wang]{Xiaoping Wang}
\address{Department of Mathematics, Hong Kong University of
Science and Technology, Hong Kong, China}
\email{mawang@ust.hk}

\title[sharp interface limit for compressible non-isentropic phase-field model]
{sharp interface limit for compressible non-isentropic phase-field model}

\keywords{compressible, non-isothermal, Navier-Stokes equation, Allen-Cahn equation, sharp interface.}

\subjclass[2010]{35Q30, 76T30, 35C20}

\date{\today}

\begin{document}
\begin{abstract}
In this paper, the sharp interface limit for the compressible non-isentropic Navier-Stokes/Allen-Cahn system is derived by the method of matched asymptotic expansion.
We show that the leading order problem satisfies the compressible Navier-Stokes equations with the interface being a free boundary. We discuss two cases in terms of different phase field diffusion coefficients. One is $M_{\epsilon}=O(1)$ and $M_{\epsilon}=O(\frac{1}{\epsilon})$, where  $\epsilon$ is the interface thickness. We have observed that the velocity and the temperature of the compressible immiscible two-phase fluids continuously through the interface. There is a jump  for the tension tensor at the interface, this jump depends on the surface tension and the mean curvature of the interface. In particular, for the first case $M_{\epsilon}=O(1)$, no matter how the density changes through the interface, the velocity of the interface in the normal direction is the same as the normal velocity of the fluid along the interface. But for the second case $M_{\epsilon}=O(\frac{1}{\epsilon})$, This phenomenon cann't occur where the density passes continuously through the interface. In fact, on this part of the interface, the normal velocity of the interface is determined by the mean curvature of the interface, the velocity and the density of the compressible immiscible two-phase fluids. That's where the phase transition happens.
\end{abstract}

\maketitle
\section{Introduction}
\label{sec:The compressible phase-field model}
%Two-phase flow refers to the flow of two phases in one system. Two-phase flow is widely used in thermal power engineering, nuclear power engineering, cryogenic engineering and aerospace engineering, etc. The theoretical analysis of two-phase flow is much more difficult than that of single-phase flow. A common treatment method is the so-called separated flow model, which holds that the concept and method of single-phase flow can be applied to each phase of the two-phase flow system,  and the interface properties  between the two phases are considered. This method is especially suitable for the description of immiscible two-phase flows, such as gas-liquid mixtures.
%An important feature of immiscible two-phase flow is the coexistence of two fluids with different phase states or components and the existence of a clear interface.
Understanding the geometry and distribution of the interface is very important for determining the immiscible two-phase flow.
The treatment of such two-phase flow's interface is derived from the idea of physicist J.D. Van der Waals  \cite{V1894}, who regarded the interface of immiscible two-phase flow as a region with a certain thickness. Mathematical models based on this idea are often called diffusion interface models, such as the famous Navier-Stokes/Allen-Cahn system, which can be used to study the immiscible two-phase flow, such as phase transformation, chemical reactions, etc., see\cite{AC1979}-\cite{K2012} and the references therein. In these literatures, by introducing diffusion interface instead of sharp interface, the authors overcome the difficulties caused by the boundary condition of interface.

For compressible immiscible two-phase flow, taking any one of the volume particles in the flow, we assume $M_i$ the mass of the components in the representative material volume $V$, $\phi_i=\frac{\rho_i}{\rho}$ the mass concentration, $\rho_i=\frac{M_i}{V}$  the apparent mass density of the fluid $i~(i=1,2)$. The total density is given by $\rho=\rho_1+\rho_2$ and $\phi=\phi_1-\phi_2$. We call $\phi$ the difference of the two
components for the fluid mixture. Obviously, $\phi$  describes the distribution of the interface. 
The compressible heat-conducting Navier-Stokes/Allen-Cahn system derived by Heida-M$\mathrm{\acute{a}}$lek-Rajagopal \cite{HMR2012} is as following
\begin{equation}\label{original NSAC}
\left\{\begin{array}{llll}
\displaystyle \partial_t\rho+\textrm{div}(\rho \mathbf{u})=0,\\
\displaystyle \partial_t(\rho \mathbf{u})+\mathrm{div}\big(\rho \mathbf{u}\otimes \mathbf{u}\big)=\mathrm{div}\mathbb{T},
  \\
\displaystyle\partial_t(\rho\phi)+\mathrm{div}\big(\rho\phi \mathbf{u}\big)=-M_{\epsilon}\mu,\\
\displaystyle\rho\mu=\rho\frac{\partial f}{\partial \phi}-\mathrm{div}\big(\rho\frac{\partial f}{\partial \nabla\phi}\big),\\
\displaystyle\partial_t(\rho E)+\mathrm{div}(\rho E\mathbf{u})=\mathrm{div}\big(\mathbb{T}\mathbf{u}+\mathrm{k}\nabla \theta-M_{\epsilon}\mu\frac{\partial f}{\partial \nabla\phi}\big),
\end{array}\right.
\end{equation}
where $ \mathbf{x}\in \Omega \subset \mathds{R}^N $, $N$ is  spatial dimension, $t>0$. The unknown functions $\rho(\mathbf{x},t)$, $\mathbf{u}(\mathbf{x},t)$, $\phi(\mathbf{x},t)$, $\theta(\mathbf{x},t)$  denote the total density, the velocity, the difference of the two components for the fluid mixture, and the absolute temperature, respectively. $\mu(\mathbf{x},t)$ is the chemical potential of the fluid. $\epsilon>0$ is the thickness of the diffuse interface. $\mathrm{k}>0$ is the coefficient of heat conduction. $M_{\epsilon}=M(\epsilon)>0$  is the  mobility coefficient.
The Cauchy stress-tensor is represented by
\begin{equation}\label{T}
\mathbb{T}=2\nu\mathbb{D}(\mathbf{u})+\lambda(\mathrm{div}\mathbf{u})\mathbb{I}-p\mathbb{I}-\rho\nabla\phi\otimes\frac{\partial f}{\partial\nabla\phi},
\end{equation}
where $\mathbb{D}\mathbf{u}$ is the  deformation tensor
\begin{equation}\label{D}
  \mathbb{D}\mathbf{u}=\frac{1}{2}\big(\nabla \mathbf{u} +\nabla^{\top} \mathbf{u}\big),
\end{equation}
and $\mathbb{I}$ is the unit matrix, $\top$ means the transpose of the matrix. $\nu>0,\lambda>0$ are  viscosity coefficients, satisfying
\begin{equation}\label{nu}
 \nu>0,\ \ \lambda+\frac{2}{N}\nu\geq0.
\end{equation}
The  total energy density $\rho E$ is given by
\begin{equation}\label{total energy density}
 \rho E=\rho \big(e+f+\frac{1}{2}\mathbf{u}^2\big),
\end{equation}
where $\rho e$ is the internal energy, $\frac{\rho\mathbf{u}^2}{2}$ is the kinetic energy,  $f$ is  the  fluid-fluid interfacial free energy density,  and it has the following form (refer to Heida-M$\mathrm{\acute{a}}$lek-Rajagopal \cite{HMR2012} and Lowengrub-Truskinovsky \cite{LT-1998}):
 \begin{equation}\label{free energy density case 2}
 f(\rho,\phi,\nabla\phi)\overset{\text{def}}{=}\frac{1}{4\epsilon\rho}(1-\phi^2)^2+\frac{\epsilon}{2\rho}|\nabla \phi|^2.
\end{equation}
Here $p=p(\rho,\theta)$, $e=e(\rho,\theta)$ and $f=f(\rho,\phi,\nabla\phi)$ obey the  second law of thermodynamics (see Lions \cite{Lions1998}),
\begin{equation}\label{second law of thermodynamics}
  ds=\frac{1}{\theta}\big(d(e+f)+pd(\frac{1}{\rho})\big),
\end{equation}
where $s$ is the thermodynamic entropy. Then deduced from \eqref{second law of thermodynamics}, we have
\begin{equation}\label{relation}
  \frac{\partial s}{\partial \theta}=\frac{1}{\theta}\frac{\partial (e+f)}{\partial \theta},\ \ \frac{\partial s}{\partial \rho}=\frac{1}{\theta}\big(\frac{\partial (e+f)}{\partial \rho}-\frac{p}{\rho^2}\big),
\end{equation}
 which implies the following compatibility equation
\begin{eqnarray}\label{pressure}
p&=&\rho^2\frac{\partial (e+f)}{\partial \rho}+\theta\frac{\partial p}{\partial \theta}\notag\\
&=&\rho^2\frac{\partial e(\rho,\theta)}{\partial \rho}-\frac{1}{4\epsilon}(\phi^2-1)^2-\frac\epsilon2|\nabla\phi|^2+\theta\frac{\partial p}{\partial \theta}.
\end{eqnarray}
Therefore, we have
\begin{equation}\label{NSAC}
\left\{\begin{array}{llll}
\displaystyle \partial_t\rho+\textrm{div}(\rho \mathbf{u})=0,\\
\displaystyle \rho \partial_t\mathbf{u}+\rho(\mathbf{u}\cdot\nabla)\mathbf{u}-2\nu\mathrm{div} \mathbb{D}\mathbf{u}-\lambda\nabla\mathrm{div}\mathbf{u}+\nabla p(\rho,\theta)=-\epsilon\mathrm{div}\big(\nabla\phi\otimes\nabla\phi\big), \\
\displaystyle \rho\partial_t\phi+\rho \mathbf{u}\nabla\phi=-M_\epsilon\mu,\\
\displaystyle \rho\mu=\frac{1}{\epsilon}(\phi^3-\phi)-\epsilon \Delta\phi,\\
\displaystyle e_\theta\big(\rho\partial_t\theta+\rho \mathbf{u}\cdot\nabla\theta\big)+\theta p_\theta\mathrm{div}\mathbf{u}-\mathrm{k} \Delta\theta=2\nu|\mathbb{D}\mathbf{u}|^2+\lambda(\mathrm{div}\mathbf{u})^2+M_{\epsilon}\mu^2.
\end{array}\right.
\end{equation}

%The study of interfacial phase field changing in two-phase immiscible  mixture  can be traced back  to the works of van der Waals \cite{V1894}, which regarded the interface between two immiscible  fluids as a layer. This idea was  successfully applied  by  Allen-Cahn \cite{AC1979} to describe the complicated phase transition phenomena in the mixture respectively.  Blesgen \cite{B1999}, Heida-Malek-Rajagopal \cite{HMR2012} added the effect of the motion of the particles  and the interaction with the diffusion into the Allen-Cahn equation,  and the thermodynamic Navier-Stokes-Allen-Cahn system were established.

About the study of compressible immiscible two-phase flow, most of the works focused on isentropic compressible problems.
Feireisl-Petzeltov$\acute{a}$-Rocca-Schimperna \cite{FPRS2010} established the global existence of finite energy weak solutions in 3-D by using the framework  which was introduced by Lions \cite{Lions1998}.  Ding-Li-Lou \cite{DLL2013}  proved the global existence of the strong solutions in 1-D with  large initial data.
Chen-Guo \cite{CG2017} generalized the result of \cite{DLL2013} to the case that the initial vacuum is allowed.
Abels-Liu \cite{AL2017}  proved the convergence of the solutions for the incompressible Stokes/Allen-Cahn system to solutions of a sharp interface model for sufficiently small times.
Witterstein \cite{W-2010} showed that the sharp-interface limit of the isentropic phase-field model is the standard two-phase compressible Navier-Stokes equations by the method of  asymptotic analysis.
Wang-Wang \cite{W-W-2007},  Xu-Di-Yu \cite{XDY2017} investigated the sharp-interface limits of the incompressible phase-field model with a generalized Navier slip boundary condition.
There is not much work for non-isentropic case. Kotschote \cite{K2012} obtained the local existence and uniqueness result for strong solutions in 3-D for compressible non-isothermal phase-field model.

%They showed that, the sharp interface limit of the incompressible two-phase immiscible fluids is the standard two-phase incompressible Navier-Stokes equations coupled with a free boundary. It should be pointed out that,  got the sharp interface limit for the compressible isothermal   Navier-Stokes-Allen-Cahn system  without the influence of boundary conditions by the method of  asymptotic analysis recently.

%In fact, in the real world, the physical size of interface thickness of the interface for immiscible two-phase flow is so small that it is impossible to conduct numerical calculation according to the real interface thickness scale, which leads to the need to artificially enlarge the thickness of the diffusion interface and calculate the approximate equation of the model
%The limit of the solution to the problem as the thickness of the diffusion interface approaches zero is very important for the reliability of the calculated results

The main purpose of this paper is to derive the sharp interface limit of the non-isentropic compressible system \eqref{NSAC} for two cases $(C1)$ and $(C2)$ by
\begin{align}
\displaystyle  (C1)\quad M_\epsilon =1,\quad & \quad (C2)\quad M_\epsilon=\frac{1}{\epsilon}.
\end{align}
The physical meaning of  the second case  $(C2)$ is that, the phase field mobility  rate is accelerated.
Before we introduce our main result, let us summarize the common symbols used in this paper.
For $\forall t>0$, suppose that unknown two-phase free interface is given by
 \begin{equation}\label{e:interface}
\Gamma(t):=\big\{\mathbf{x}\in\Omega\big| \phi_{\epsilon}(\mathbf{x},t)=0\big\}.
 \end{equation}
Obviously, $\Gamma(t)$ divides the whole domain $ \Omega $ into two separated domain $\Omega^-(t)$ and $\Omega^+(t) $ which represented the domains occupied by fluid 1 and  fluid 2 respectively, more precisely
 \begin{equation}\label{Omega+(t) and Omega-(t)}
 \Omega^-(t)=\big\{\mathbf{x}\in\Omega\big|\phi_{\epsilon}(\mathbf{x},t)<0\big\},\qquad\Omega^+(t)=\big\{\mathbf{x}\in\Omega\big|:\phi_{\epsilon}(\mathbf{x},t)>0\big\},\notag
 \end{equation}
and
\begin{equation}
 \Omega=\Omega^-(t)\cup\Gamma(t)\cup\Omega^+(t), \quad \forall t>0.\notag
\end{equation}
Therefore, we have the following definition
 \begin{equation}\label{Omegasign}
 \Omega^-=\big\{(\mathbf{x},t)\big|\mathbf{x}\in\Omega^-(t)\big\},\quad\Omega^+=\big\{(\mathbf{x},t)\big|:\mathbf{x}\in\Omega^+(t)\big\},\quad \Gamma=\big\{(\mathbf{x},t)\big|\mathbf{x}\in\Gamma(t)\big\}.
 \end{equation}
We use the method of matched asymptotic expansion to determine the sharp interface limit as $ \epsilon \rightarrow 0 $.
Unlike incompressible fluids, the density of compressible fluids varies. Especially for immiscible two-phase flow, the change of density may cause phase transition. To illustrate the mass density properties near the interface in more detail, we define the set $ S \subset \Gamma $ by
 \begin{equation}\label{S}
  S=\{(\mathbf{x},t)\in \Gamma| \rho^+(\mathbf{x},t)=\rho^-(\mathbf{x},t)\},
 \end{equation}
which implies that the jump of the density $\rho$ can't occur on this part of the interface $ \Gamma $.
Now, we give the main theorem of sharp interface limit.

\begin{theorem}\label{thm-main} We assume \eqref{e:interface}-\eqref{Omegasign}. Let $(\rho,\mathbf{u},\phi,\mu,\theta) $ be a solution of the Navier-Stokes-Allen-Cahn system \eqref{NSAC}. We assume that an outer asymptotic expansion, that is
\begin{align*}
&\rho=\rho_0+\eps \rho_1+\cdots,\quad \mathbf{u}=\mathbf{u}_0+\eps \mathbf{u}_1+\cdots,\quad
\phi=\phi_0+\eps \phi_1+\cdots,\\
&\mu=\eps^{-1}\mu_0+\mu_1+\cdots,\quad \theta =\theta_0+\eps \theta_1+\cdots,
\end{align*}
and an inner asymptotic expansion, that is
\begin{align*}
&\rho=\widetilde{\rho}_0+\eps \widetilde{\rho_1}+\cdots,\quad \mathbf{u}=\widetilde{\mathbf{u}}_0+\eps \widetilde{\mathbf{u}}_1+\cdots,\quad \phi=\widetilde{\phi}_0+\eps \widetilde{\phi}_1+\cdots,\\
&\widetilde{\mu}=\eps^{-1}\widetilde{\mu}_0+\widetilde{\mu}_1+\cdots,\quad \widetilde{\theta} =\widetilde{\theta}_0+\eps \widetilde{\theta}_1+\cdots,
\end{align*}
for $ \rho, \mathbf{u}, \phi, \mu, \theta $. We suppose that $ \rho_0>0 $ and $ \widetilde{\rho}_0>0 $. Then as $ \eps \rightarrow 0$, the system \eqref{NSAC} converges to the sharp interface problem
\begin{equation}\label{free boundary problem for NSAC}
\left\{\begin{array}{llll}
\displaystyle  \partial_t \rho + \mathrm{div}(\rho  \mathbf{u} )=0,& \mathrm{in}\ \Omega^\pm,\\
\displaystyle\rho \partial_t\mathbf{u}+\rho (\mathbf{u}\cdot \nabla)\mathbf{u}-2\nu\mathrm{div} \mathbb{D}\mathbf{u}-\lambda\nabla\mathrm{div}\mathbf{u}+\nabla\big( \rho^2\frac{\partial e}{\partial \rho}+\theta\frac{\partial p}{\partial \theta}\big)=0,& \mathrm{in}\ \Omega^\pm,\\
e_\theta\big(\rho\partial_t\theta+\rho \mathbf{u}\cdot\nabla\theta\big)+\theta p_\theta\mathrm{div}\mathbf{u}-\mathrm{k} \Delta \theta=2\nu|\mathbb{D}\mathbf{u}|^2+\lambda(\mathrm{div}\mathbf{u})^2,& \mathrm{in}\ \Omega^\pm,\\
\displaystyle \phi=\pm1, & \mathrm{in}\ \Omega^\pm,
\end{array}\right.
\end{equation}
with the jump conditions
\begin{equation}\label{jump condition-case-1}
 (C1)\quad M_\epsilon=1:\ \left\{\begin{array}{llll}
 \displaystyle [\mathbf{u}]_\Gamma=0,\quad [\theta]_\Gamma=0, & \mathrm{on}~~\Gamma,\\
 \displaystyle V_{\mathbf{n}}-\mathbf{u} \cdot \mathbf{n}  =0,& \mathrm{on}~~\Gamma,
 \\
 \displaystyle \big[\big(2\nu\mathbb{D}\mathbf{u}+\lambda\mathrm{div}\mathbf{u}\mathbb{I}-\big( \rho^2\frac{\partial e}{\partial \rho}+\theta\frac{\partial p}{\partial \theta}\big)\mathbb{I}\big)\mathbf{n}\big]_\Gamma=\sigma \kappa \mathbf{n},& \mathrm{on}~~\Gamma,
\end{array}
\right.
\end{equation}
\begin{equation}\label{jump condition-case-2}
 \ (C2)\quad M_\epsilon=\frac{1}{\epsilon}:\ \left\{\begin{array}{llll}
 \displaystyle [\mathbf{u}]_\Gamma=0,\quad [\theta]_\Gamma=0,& \mathrm{on}~~\Gamma,\\
 \displaystyle [\rho]_\Gamma =0,\quad \rho^2(V_\mathbf{n}-\mathbf{u} \cdot \mathbf{n}) =-\kappa,& \mathrm{on}~~S,
 \\
 \displaystyle V_{\mathbf{n}}-\mathbf{u} \cdot \mathbf{n} =0=\kappa,& \mathrm{on}~~\Gamma\setminus S,
 \\
 \displaystyle \big[\big(2\nu\mathbb{D}\mathbf{u}+\lambda\mathrm{div}\mathbf{u}\mathbb{I}-\big( \rho^2\frac{\partial e}{\partial \rho}+\theta\frac{\partial p}{\partial \theta}\big)\mathbb{I}\big)\mathbf{n}\big]_\Gamma=\sigma \kappa \mathbf{n},& \mathrm{on}~~\Gamma,
\end{array}
\right.
\end{equation}
on the interface $ \Gamma $, where $\Gamma$ is defined as \eqref{Omegasign} and $S$ defined as \eqref{S}. $ [\cdot]_\Gamma=\cdot^+-\cdot^-$ denotes the jump of limiting values across the interface, $\sigma$ is  the constant coefficient of surface tension (see \eqref{sigma0}), $\kappa $ the curvature of the interface $ \Gamma $, $\mathbf{n}$ the unit normal of the interface pointing to $\Omega^+$, $ V_\mathbf{n} $ denotes the normal velocity of the interface $ \Gamma $.
\end{theorem}

\begin{remark}
Theorem 1.1 shows that the leading order problem satisfies the compressible Navier-Stokes equations with the interface being a free boundary.
%for two different mobility coefficient $M_{\epsilon}$ of phase field ($M_{\epsilon}=O(1)$ and $M_{\epsilon}=O(\frac{1}{\epsilon})$, where  $\epsilon$ is the interface thickness).
Whether the diffusion coefficient $M_{\epsilon}=O(1)$ or $M_{\epsilon}=O(\frac{1}{\epsilon})$, the velocity and the temperature of the flow continuously through the interface. There is a jump  for the tension tensor at the interface, and this jump depends on the surface tension and the mean curvature of the interface. In particular, for the first case $M_{\epsilon}=O(1)$, no matter how the density changes through the interface, the velocity of the interface in the normal direction is the same as the normal velocity of the fluid along the interface. But for the second case $M_{\epsilon}=O(\frac{1}{\epsilon})$, this phenomenon can only occur at the part of the interface where the density is discontinuous, and the mean curvature for this part of the interface is zero. In other parts of the interface, the normal velocity of the interface is determined by the velocity, density of the fluid and the curvature of the surface, and that's where the phase transition happens.
\end{remark}

%\begin{remark}
%In the case $(C_1)\ M=1$,  the condition \eqref{jump condition-case-1}$_1$ means the velocity and the temperature of the two-phase flow are continuous across the interface between the phases. \eqref{jump condition-case-1}$_2$ means the velocity of the interface  equal to the velocity of flow in the normal direction, and the jump of the density $\rho$ on the interface $\Gamma$ is feasible but not required. The condition \eqref{jump condition-case-1}$_3$ means that there is a jump for the tension tensor  across the interface between the phases, and this jump depends on the surface tension and the mean curvature of the interface.
%\end{remark}

%\begin{remark}
%In the case $ (C_2)\ M=\frac{1}{\epsilon}$,  the behavior of the mass density at the interface $\Gamma$ depends on the relative velocity between the interface and the immiscible mixture flow. \eqref{jump condition-case-2} indicates that, the phase transition may occurs near the interface if and only if the mobility rate is accelerated for \eqref{NSAC}. Especially on the $S$ where the jump of the density $\rho$ can't occur, and the velocity of the interface is inconsistent with the velocity of the fluid, which satisfies the Gibbs-Thomson condition \eqref{jump condition-case-2}$_2$. For another, on the $S$, where the interfacial velocity is consistent with the fluid velocity, and the curvature of the interface $ \Gamma $ is zero.
%\end{remark}

\section{Asymptotic expansion for non-isentropic compressible phase-field model} \label{sec:expansion}% (fold)
In this section we consider the asymptotic analysis for \eqref{NSAC} for two cases $(C1)$ and $(C2)$.
The analysis process is as follows, firstly we consider the outer expansions far from the interface $ \Gamma $, and then we do the inner expansions near the interface $ \Gamma $, and finally we combine them together to obtain the sharp interface limit of the system \eqref{NSAC} in $ \Omega $.
To make the presentation in this section clear, we use $\rho_{\epsilon}$, $\mathbf{u}_{\epsilon}$, $\phi_{\epsilon}$, $\mu_{\epsilon}$ and $\theta_{\epsilon}$ instead of $\rho$, $\mathbf{u}$, $\phi$, $\mu$ and $\theta$ in the system \eqref{NSAC}, to show explicitly that these functions depend on $\epsilon$.
\subsection{Outer expansion} % (fold)\label{sub:Outer expansion}
Far from the two-phase interface $\Gamma$, we use the following ansatz,
\begin{equation}\label{outer-e}
\begin{array}{l}
\displaystyle\rho_{\epsilon}^{\pm}=\rho_0^{\pm}+\epsilon\rho_1^{\pm}+\epsilon^2 \rho_2^{\pm}+\cdots,\\
\displaystyle\mathbf{u}_{\epsilon}^{\pm}=\mathbf{u}_0^{\pm}+\epsilon\mathbf{u}_1^{\pm}+\epsilon^2 \mathbf {u}_2^{\pm}+\cdots,\\
\displaystyle\phi_{\epsilon}^{\pm}=\phi_0^{\pm}+\epsilon\phi_1^{\pm}+\epsilon^2 \phi_2^{\pm}+\cdots,\\
\displaystyle\mu_{\epsilon}^{\pm}=\eps^{-1}\mu_0^{\pm}+\mu_1^{\pm}+\epsilon \mu_2^{\pm}+\cdots,\\
\displaystyle\theta_{\epsilon}^{\pm}=\theta_0^{\pm}+\epsilon\theta_1^{\pm}+\epsilon^2 \theta_2^{\pm}+\cdots.
\end{array}
\end{equation}
Here $g^{\pm}$ denotes the restriction of a function $g$ in $\Omega^+$ and $\Omega^-$ respectively. Since $\rho >0, \theta>0$, we presume that $\rho_0^{\pm}>0, \theta>0$.

Combined with the above analysis, let us first plug ansatz \eqref{outer-e} in the mass conservation equation \eqref{NSAC}$_1$ and compare the coefficients of terms with the same power of $ \epsilon$, and we achieve on the lowest order
\begin{align}
\displaystyle \mathcal{O}(1): \quad \partial_t \rho_0^{\pm} + \textrm{div}(\rho_0^{\pm}  \mathbf{u}_0^{\pm} )=0. \label{rho-outer}
\end{align}
Similarly, plugging ansatz \eqref{outer-e} in the momentum conservation equation \eqref{NSAC}$_2$, we obtain
\begin{align}
\displaystyle \mathcal{O}(\epsilon^{-1}):& \quad \nabla(\frac{1}{4}(1-(\phi_0^\pm)^2)^2)=0, \label{uu-outer1}\\
\displaystyle \mathcal{O}(1):& \quad \rho_0^{\pm} \partial_t\mathbf{u}_0^{\pm}+\rho_0^{\pm} (\mathbf{u}_0^{\pm}\cdot \nabla) \mathbf{u}_0^{\pm}-\nabla (((\phi_0^\pm)^3-\phi_0^\pm)\phi_1^{\pm})=\textrm{div}\mathbb{S}_0^{\pm}-\nabla p_0^{\pm}, \label{uu-outer2}
\end{align}
where
\begin{equation}\label{v-stress-outer}
  \mathbb{S}_0^{\pm}=2\nu\mathbb{D}\mathbf{u}_0^{\pm}+\lambda\mathrm{div}\mathbf{u}_0^{\pm}\mathbb{I},
\end{equation}
and
\begin{equation}\label{pressure-outer}
   p_0^{\pm}=(\rho_0^{\pm})^2e_\rho(\rho_0^{\pm},\theta_0^{\pm})+\theta_0^{\pm}p_\theta(\rho_0^{\pm},\theta_0^{\pm}).
\end{equation}
Next, we plug ansatz \eqref{outer-e} in conservation equation of phase field \eqref{NSAC}$_{3}$ to achieve
\begin{align}
\displaystyle \mathcal{O}(\epsilon^{\chi_{M_\epsilon}-2}):& \quad \mu_0^{\pm}=0, \label{phi-outer1}\\
\displaystyle \mathcal{O}(\epsilon^{\chi_{M_\epsilon}-1}):& \quad \chi_{M_\epsilon}(\rho_0^{\pm}\partial_t\phi_0^{\pm}+\rho_0^{\pm} \mathbf{u}_0^{\pm}\cdot
\nabla\phi_0^{\pm}) =-\mu_1^{\pm}, \label{phi-outer2}
\end{align}
where
\begin{equation}\label{chi-eta}
\chi_{M_\epsilon}=
\left\{\begin{array}{llll}
\displaystyle  1,&\ \mathrm{for~~(C1)}\ M_\epsilon=1, \\
\displaystyle 0,&\ \mathrm{for~~(C2)}\ M_\epsilon=\frac{1}{\epsilon}.
\end{array}\right.
\end{equation}
Plugging ansatz \eqref{outer-e} in potential equation \eqref{NSAC}$_4$ and comparing the coefficients of terms with the same power of $ \epsilon$, we obtain
\begin{align}
\displaystyle \mathcal{O}(\eps^{-1}): \quad & \rho_0^{\pm}\mu_0^{\pm}=\phi_0^\pm\big((\phi_0^\pm)^2-1\big).\label{mu-outer1}
\end{align}
Finally, plugging ansatz \eqref{outer-e} in energy equation \eqref{NSAC}$_5$, we obtain
\begin{align}
& \displaystyle \mathcal{O}(1):  e_\theta(\rho_0^{\pm},\theta_0^{\pm})(\rho_0^{\pm}\partial_t\theta_0^{\pm}+\rho_0^{\pm}\mathbf{u}_0^{\pm} \cdot\nabla \theta_0^{\pm})+\theta_0^{\pm}p_\theta(\rho_0^{\pm},\theta_0^{\pm})\mathrm{div}\mathbf{u}_0^{\pm}-\mathrm{k} \Delta \theta_0^{\pm}\nonumber\\
& =2 \nu|\mathbb{D}\mathbf{u}_0^{\pm}|^2\!+\lambda(\mathrm{div}\mathbf{u}_0^{\pm})^2\!+\chi_{M_\epsilon}\big((\mu_1^{\pm})^2\!+2\mu_0^{\pm}\mu_2^{\pm}\big)\!+(1\!-\!\chi_{M_\epsilon})\big(2\mu_1^{\pm}\mu_2^{\pm}\!+\!2\mu_0^{\pm}\mu_3^{\pm}\big),\label{theta-outer1}
\end{align}
where $\chi_{M_\epsilon}$ is defined in \eqref{chi-eta}.
Then from the equation \eqref{uu-outer1}, \eqref{phi-outer1}, \eqref{phi-outer2} and \eqref{mu-outer1}, we have
\begin{align}\label{phi0}
\displaystyle  \phi_0^{\pm}=\pm 1,\quad \mu_0^{\pm}=0,\quad \mu_1^{\pm}=0, \qquad \text{in}~~\Omega^\pm.
\end{align}
Combining \eqref{uu-outer2} and \eqref{mu-outer1}, it follows that
\begin{align}\label{uu-outer}
\displaystyle  \rho_0^{\pm}\partial_t \mathbf{u}_0^{\pm}+\rho_0^{\pm} (\mathbf{u}_0^{\pm} \cdot \nabla) \mathbf{u}_0^{\pm}+\nabla p_0^\pm=\textrm{div}\mathbb{S}_0^{\pm}.
\end{align}
Moreover, \eqref{theta-outer1} and \eqref{phi0} derive that
\begin{align}
\displaystyle &  e_\theta(\rho_0^{\pm},\theta_0^{\pm})(\rho_0^{\pm}\partial_t\theta_0^{\pm}+\rho_0^{\pm}\mathbf{u}_0^{\pm} \cdot\nabla \theta_0^{\pm})+\theta_0^{\pm}p_\theta(\rho_0^{\pm},\theta_0^{\pm})\mathrm{div}\mathbf{u}_0^{\pm}-\mathrm{k} \Delta \theta_0^{\pm}\nonumber\\
& =2 \nu|\mathbb{D}\mathbf{u}_0^{\pm}|^2+\lambda(\mathrm{div}\mathbf{u}_0^{\pm})^2. \label{theta-outer}
\end{align}
By using \eqref{rho-outer}, \eqref{phi0}, \eqref{uu-outer} and \eqref{theta-outer}, we can state the following lemma.
\begin{lemma}\label{lem-outer}
Assume $ \rho_0^\pm>0 $. Letting $ \epsilon\rightarrow 0 $ in ansatz \eqref{outer-e}, we derive
from \eqref{rho-outer}, \eqref{uu-outer2}, \eqref{phi-outer1}, \eqref{phi-outer2}, \eqref{mu-outer1} and \eqref{theta-outer1}, in the regions outside the transition layer, to the problem
\begin{align}
&  \partial_t \rho^{\pm} + \mathrm{div}(\rho^{\pm}  \mathbf{u}^{\pm} )=0,& \mathrm{in}\ \Omega^\pm,\label{rho-inter}\\
&  \rho^{\pm}\partial_t \mathbf{u}^{\pm}+\rho^{\pm}(\mathbf{u}^{\pm}\cdot\nabla) \mathbf{u}^{\pm}-\mathrm{div} (2\nu\mathbb{D}\mathbf{u}^{\pm}+\lambda\mathrm{div}\mathbf{u}^{\pm}\mathbb{I})\notag\\
&\qquad+\nabla \big((\rho^\pm)^2(e_\rho(\rho^\pm,\theta^\pm)+\theta^\pm p_\theta(\rho^\pm,\theta^\pm)\big)=0,& \mathrm{in}\ \Omega^\pm, \label{uu-inter}\\
& e_\theta(\rho^{\pm},\theta^{\pm})(\rho^{\pm}\partial_t\theta^{\pm}+\rho^{\pm}\mathbf{u}^{\pm} \cdot\nabla \theta^{\pm})+\theta^{\pm}p_\theta(\rho^{\pm},\theta^{\pm})\mathrm{div}\mathbf{u}^{\pm}-\mathrm{k} \Delta \theta^{\pm}\nonumber\\
&\quad =2 \nu|\mathbb{D}\mathbf{u}^{\pm}|^2+\lambda(\mathrm{div}\mathbf{u}^{\pm})^2 \label{theta-inter},& \mathrm{in}\ \Omega^\pm,
\end{align}
and $\phi_0^{\pm}=\pm 1$ in $ \Omega^\pm $.
\end{lemma}

\subsection{Inner expansion}\label{innersection}
In this subsection, we propose to analysis the inner expansion near the interface $\Gamma$.
Let $d(\mathbf{x},t)$ be signed distance to $\Gamma$, which is well-defined near the interface. Then the unit normal of the interface pointing to $\Omega^+$ is given by $\mathbf{n}=\nabla d$ and the normal velocity of the interface $ \Gamma $ is given by $ V_\mathbf{n}=-\partial_t d $.
We introduce a new rescaled variable
\begin{equation}\label{d}
  \xi=\frac{d(\mathbf{x},t)}{\epsilon}.
\end{equation}
For any function $g(\mathbf{x},t)$ (e.g. $g=\rho_\epsilon,\mathbf{u}_\epsilon, \phi_\epsilon, \mu_\epsilon, \theta_\epsilon$), we can rewrite it as
\begin{equation}\label{g}
g(\mathbf{x},t)=\widetilde{g}(\mathbf{x},t,\xi).
\end{equation}
Then we have
\begin{equation}\label{new-d}
\begin{array}{l}
\nabla g=\nabla\widetilde{g}+\epsilon^{-1}\partial_\xi\widetilde{g}\mathbf{n},\\
\Delta g=\Delta \widetilde{g}+\epsilon^{-1}\partial_\xi\widetilde{g}\kappa+2\epsilon^{-1}(\mathbf{n}\cdot\nabla)\partial_{\xi}\widetilde{g}
+\epsilon^{-2}\partial_{\xi\xi}\widetilde{g},\\
\partial_t g=\partial_t\widetilde{g}-\epsilon^{-1} \partial_\xi\widetilde{g} V_\mathbf{n}.
\end{array}
\end{equation}
Here we use the fact that $\nabla\cdot\mathbf n=\kappa$,  the mean curvature of the
interface. $\kappa(\mathbf{x})$ for $\mathbf{x}\in\Gamma(t)$  is positive (resp. negative) if the domain $\Omega_-$ is convex (resp. concave) near $\mathbf{x}$.

In the inner region,  we assume that
\begin{equation}\label{inner-e}
\begin{array}{l}
\widetilde\rho_{\epsilon}=\widetilde\rho_0+\epsilon\widetilde\rho_1+\epsilon^2 \widetilde\rho_2+\cdots,\\
\widetilde{\mathbf{u}}_{\epsilon}=\widetilde{\mathbf{u}}_0+\epsilon\widetilde{\mathbf{u}}_1+\epsilon^2 \widetilde{\mathbf u}_2+\cdots,\\
\widetilde\phi_{\epsilon}=\widetilde\phi_0+\epsilon\widetilde\phi_1+\epsilon^2 \widetilde\phi_2+\cdots,\\
\widetilde\mu_{\epsilon}=\epsilon^{-1}\widetilde \mu_0+\widetilde \mu_1+\epsilon\widetilde \mu_2+\cdots,\\
\widetilde\theta_{\epsilon}=\widetilde \theta_0+\epsilon\widetilde \theta_1+\epsilon^2\widetilde \theta_2+\cdots.
\end{array}
\end{equation}
In the next we represent the system \eqref{NSAC} in the new coordinates and compare the order of the $ \epsilon$ coefficients.
Let us first plug ansatz \eqref{inner-e} in the mass conservation equation \eqref{NSAC}$_1$ to infer the following equation in new coordinates
\begin{align}
\displaystyle  \partial_t \widetilde{\rho}_0-\eps^{-1}\partial_\xi\widetilde{\rho}_0 V_n+\textrm{div}(\widetilde{\rho}_0 \widetilde{\mathbf{u}}_0)+\epsilon^{-1}\partial_\xi(\widetilde{\rho}_0 \widetilde{\mathbf{u}}_0) \cdot \mathbf{n}+\mathcal{O}(\epsilon)=0.
\end{align}
By comparing the coefficients of same $\epsilon$-order, we obtain
\begin{align}
\displaystyle \mathcal{O}(\epsilon^{-1}):&  \quad -\partial_\xi\widetilde{\rho}_0 V_\mathbf{n}+\partial_\xi(\widetilde{\rho}_0 \widetilde{\mathbf{u}}_0) \cdot \mathbf{n}=0,\label{rho-inner1}\\
\displaystyle \mathcal{O}(1):&  \quad \partial_t \widetilde{\rho}_0+\textrm{div}(\widetilde{\rho}_0 \widetilde{\mathbf{u}}_0)-
\partial_\xi\widetilde{\rho}_1 V_\mathbf{n}+ \partial_{\xi}(\widetilde{\rho}_0 \widetilde{\mathbf{u}}_1+\widetilde{\rho}_1 \widetilde{\mathbf{u}}_0)\cdot \mathbf{n}=0. \label{rho-inner2}
\end{align}
Next, we evaluate the right-hand-side of the momentum conservation equation \eqref{NSAC}$_2$. By using \eqref{new-d}, through  differentiation, we obtain
\begin{align}\label{div-s-exp}
&\mathop{\mathrm{div}}\nolimits(\nu(\nabla \mathbf{u}+\nabla ^T\mathbf{u}))+\mathop{\mathrm{div}}\nolimits(\lambda\mathop{\mathrm{div}}\nolimits\mathbf{u}\mathbb{I})\notag\\
&=\frac{1}{\epsilon^2}\partial_\xi\big((\nu+\lambda)\partial_\xi\widetilde{\mathbf{u}}_0\cdot \mathbf{n}\big)\mathbf{n}+ \frac{1}{\epsilon^2}\partial_\xi\big(\nu\partial_\xi\widetilde{\mathbf{u}}_0\big) +\frac{1}{\epsilon}\partial_\xi\big(\nu(\nabla \widetilde{\mathbf{u}}_0+\nabla^{T} \widetilde{\mathbf{u}}_0)\big)\cdot \mathbf{n}
\nonumber \\
&\qquad +\frac{1}{\epsilon}\mathop{\mathrm{div}}\nolimits(\nu \partial_\xi \widetilde{\mathbf{u}}_0 \otimes \mathbf{n})+\frac{1}{\epsilon}\nu\partial_\xi \widetilde{\mathbf{u}}_0 \kappa+\frac{1}{\epsilon}\partial_\xi\big(\lambda\mathop{\mathrm{div}}\nolimits \widetilde{\mathbf{u}}_0\big)\cdot \mathbf{n} \nonumber\\
& \qquad+\frac{1}{\epsilon}\nabla \big(\lambda\partial_\xi \widetilde{\mathbf{u}}_0 \cdot \mathbf{n}\big)
+\frac{1}{\epsilon}(\nu+\lambda)(\partial_\xi \widetilde{\mathbf{u}}_0 \cdot \nabla)\mathbf{n}+\mathcal{O}(1) ,
\end{align}
and
\begin{align}\label{expansion of p}
 & \nabla p(\rho,\theta)+\epsilon\mathop{\mathrm{div}}\nolimits(\nabla \phi \otimes \nabla\phi)\nonumber \\
 &= \frac{1}{\epsilon}\partial_\xi p_0 \mathbf{n}-\frac{1}{\epsilon}\nabla ({\tilde{\phi}_0}^3-\widetilde{\phi}_0) - \frac{1}{4\epsilon^2}\partial_\xi ({\tilde{\phi}_0}^2-1)^2 \mathbf{n} +\frac{1}{2\eps^2}\partial_\xi|\partial_\xi \widetilde{\phi}_0|^2 \mathbf{n}\nonumber \\
 &\quad +\frac{1}{\epsilon} \mathbf{n}\cdot \nabla|\partial_\xi \widetilde{\phi}_0|^2 \mathbf{n}
+\frac{1}{\epsilon}|\partial_\xi \widetilde{\phi}_0|^2 \kappa \mathbf{n}
+\frac{1}{\epsilon}\partial_{\xi\xi}\widetilde{\phi}_0 \nabla\widetilde{\phi}_0 +\mathcal{O}(1),
\end{align}
where $ p_0=\tilde{\rho}_0^2e_\rho(\tilde{\rho}_0,\tilde{\theta}_0)+\tilde{\theta}_0p_\theta(\tilde{\rho}_0,\tilde{\theta}_0)$.
Substituting \eqref{div-s-exp} and \eqref{expansion of p} into  \eqref{NSAC}$_2$, we have
\begin{align}
& -\frac{1}{\epsilon}\widetilde{\rho}_0\partial_\xi \widetilde{\mathbf{u}}_0V_n+ \frac{1}{\epsilon}\widetilde{\rho}_0 \widetilde{\mathbf{u}}_0 \cdot\mathbf{n} \partial_\xi \widetilde{\mathbf{u}}_0\nonumber \\
&=\frac{1}{\epsilon^2}\partial_\xi\big((\nu+\lambda)\partial_\xi\widetilde{\mathbf{u}}_0\cdot \mathbf{n}\big)\mathbf{n}+ \frac{1}{\epsilon^2}\partial_\xi\big(\nu\partial_\xi\widetilde{\mathbf{u}}_0\big)+\frac{1}{\epsilon}\partial_\xi\big(\nu(\nabla \widetilde{\mathbf{u}}_0+\nabla^{T} \widetilde{\mathbf{u}}_0)\big)\cdot \mathbf{n} \nonumber\\
&\quad+\frac{1}{\epsilon}\mathop{\mathrm{div}}\nolimits(\nu \partial_\xi \widetilde{\mathbf{u}}_0 \otimes \mathbf{n}) +\frac{1}{\epsilon}\nu \partial_\xi \widetilde{\mathbf{u}}_0 \kappa+\frac{1}{\epsilon}\partial_\xi\big(\lambda\mathop{\mathrm{div}}\nolimits \widetilde{\mathbf{u}}_0\big)\cdot \mathbf{n} +\frac{1}{\epsilon}\nabla \big(\lambda\partial_\xi \widetilde{\mathbf{u}}_0 \cdot \mathbf{n}\big)\nonumber\\
&\quad +\frac{1}{\epsilon}(\nu+\lambda)(\partial_\xi \widetilde{\mathbf{u}}_0 \cdot \nabla)\mathbf{n}-\frac{1}{\epsilon}\partial_\xi p_0 \mathbf{n}+\frac{1}{4\epsilon}\nabla ({\tilde{\phi}_0}^2-1)^2 + \frac{1}{4\epsilon^2}\partial_\xi ({\tilde{\phi}_0}^2-1)^2 \mathbf{n}\nonumber\\
&\quad-\frac{1}{2\epsilon^2}\partial_\xi|\partial_\xi \widetilde{\phi}_0|^2 \mathbf{n}
-\frac{1}{\epsilon}\mathbf{n}\cdot \nabla|\partial_\xi \widetilde{\phi}_0|^2 \mathbf{n}
-\frac{1}{\epsilon}|\partial_\xi \widetilde{\phi}_0|^2 \kappa \mathbf{n}
-\frac{1}{\epsilon}\partial_{\xi\xi}\widetilde{\phi}_0 \nabla\widetilde{\phi}_0+\mathcal{O}(1).
\end{align}
By comparing the coefficients of same $\epsilon$-order, we achieve
\begin{align}
\displaystyle \mathcal{O}(\epsilon^{-2}): &\quad \partial_\xi\big((\nu+\lambda)\partial_\xi\widetilde{\mathbf{u}}_0\cdot \mathbf{n}\big)\mathbf{n}+ \partial_\xi\big(\nu\partial_\xi\widetilde{\mathbf{u}}_0\big)+\frac14\partial_\xi ({\tilde{\phi}_0}^2-1)^2\mathbf{n}-\frac{1}{2}\partial_\xi|\partial_\xi \widetilde{\phi}_0|^2 \mathbf{n}=0,\label{uu-inner1}\\
\displaystyle \mathcal{O}(\epsilon^{-1}): &\quad -\widetilde{\rho}_0\partial_\xi \widetilde{\mathbf{u}}_0V_n+ \widetilde{\rho}_0 \widetilde{\mathbf{u}}_0 \cdot\mathbf{n} \partial_\xi \widetilde{\mathbf{u}}_0 +\partial_\xi p_0 \mathbf{n}\nonumber\\
&=\partial_\xi\big(\nu(\nabla \widetilde{\mathbf{u}}_0+\nabla^{T} \widetilde{\mathbf{u}}_0)\big)\cdot \mathbf{n} +\mathop{\mathrm{div}}\nolimits(\nu \partial_\xi \widetilde{\mathbf{u}}_0 \otimes \mathbf{n})
+\nu \partial_\xi \widetilde{\mathbf{u}}_0 \kappa \nonumber\\
& +\partial_\xi\big(\lambda\mathop{\mathrm{div}}\nolimits \widetilde{\mathbf{u}}_0\big)\cdot \mathbf{n}
+\nabla \big(\lambda\partial_\xi \widetilde{\mathbf{u}}_0 \cdot \mathbf{n}\big)
+(\nu+\lambda)(\partial_\xi \widetilde{\mathbf{u}}_0 \cdot \nabla)\mathbf{n} \nonumber\\
&+\frac14\nabla ({\tilde{\phi}_0}^2-1)^2+\partial_\xi (({\tilde{\phi}_0}^3-\widetilde{\phi}_0)\widetilde{\phi}_1)\mathbf{n}-\mathbf{n}\cdot\nabla|\partial_\xi \widetilde{\phi}_0|^2 \mathbf{n}-|\partial_\xi \widetilde{\phi}_0|^2 \kappa \mathbf{n} \nonumber\\
&-\partial_{\xi\xi}\widetilde{\phi}_0 \nabla\widetilde{\phi}_0+\partial_\xi\big((\nu+\lambda)\partial_\xi\hat{\mathbf{u}}_1\cdot \mathbf{n}\big)\mathbf{n}+ \partial_\xi\big(\nu\partial_\xi\hat{\mathbf{u}}_1\big)-\partial_\xi\big(\partial_\xi \widetilde{\phi}_0\partial_\xi \widetilde{\phi}_1\big)\mathbf{n}. \label{uu-inner2}
\end{align}
Similarly, equation \eqref{NSAC}$_3$ can be rewritten in the new coordinates as
\begin{align}
& -\frac{1}{\epsilon}\widetilde{\rho}_0 \partial_\xi \widetilde{\phi}_0 V_\mathbf{n}+\frac{1}{\epsilon}\widetilde{\rho}_0 \widetilde{\mathbf{u}}_0 \cdot \mathbf{n}\partial_\xi \widetilde{\phi}_0=-M_\epsilon\frac{1}{\epsilon}\widetilde{\mu}_0+\mathcal{O}(1).
\end{align}
For (C1) $M_\epsilon=1$, we compare the coefficients of same $\epsilon$-order to get
\begin{align}
\displaystyle \mathcal{O}(\epsilon^{-1}):&  \quad -\widetilde{\rho}_0 \partial_\xi \widetilde{\phi}_0 V_n+\widetilde{\rho}_0 \widetilde{\mathbf{u}}_0\cdot \mathbf{n} \partial_\xi \widetilde{\phi}_0 =-\widetilde{\mu}_0. \label{phi-inner2}
\end{align}
For (C2) $M_\epsilon=\frac{1}{\epsilon}$, we have
\begin{align}
\displaystyle \mathcal{O}(\epsilon^{-2}):&  \quad \widetilde{\mu}_0=0,\label{phi-inner2-1}\\
\displaystyle \mathcal{O}(\epsilon^{-1}):&  \quad -\widetilde{\rho}_0 \partial_\xi \widetilde{\phi}_0 V_n+\widetilde{\rho}_0 \widetilde{\mathbf{u}}_0\cdot \mathbf{n} \partial_\xi \widetilde{\phi}_0 =-\widetilde{\mu}_1. \label{phi-inner2-2}
\end{align}
Plugging ansatz \eqref{inner-e} in the potential equation \eqref{NSAC}$_4$, we obtain
\begin{align}
\frac{1}{\epsilon}\widetilde{\rho}_0 \widetilde{\mu}_0=-2(\mathbf{n}\cdot\nabla) \partial_\xi\widetilde{\phi}_0-\partial_\xi \widetilde{\phi}_0\kappa-\frac{1}{\epsilon}\partial_{\xi\xi} \widetilde{\phi}_0+\frac{1}{\epsilon}({\tilde{\phi}_0}^3-\widetilde{\phi}_0)+\mathcal{O}(\eps),
\end{align}
comparing the coefficients of same $\eps$-order, we achieve
\begin{align}
\displaystyle \mathcal{O}(\eps^{-1}):&  \quad \widetilde{\rho}_0 \widetilde{\mu}_0=-\partial_{\xi\xi} \widetilde{\phi}_0+{\tilde{\phi}_0}^3-\widetilde{\phi}_0,\label{mu-inner1}\\
\mathcal{O}(1):&  \quad \widetilde{\rho}_0\widetilde{\mu}_1+\widetilde{\rho}_1 \widetilde{\mu}_0=-2(\mathbf{n}\cdot\nabla) \partial_\xi\widetilde{\phi}_0-\partial_\xi \widetilde{\phi}_0\kappa-\partial_{\xi\xi} \widetilde{\phi}_1+(3{\tilde{\phi}_0}^2-1)\widetilde{\phi}_1.\label{mu-inner2}
\end{align}
Finally, plugging ansatz \eqref{inner-e} in the energy equation \eqref{NSAC}$_5$.
For (C1) $M_\epsilon=1$, we compare the coefficients of same $\epsilon$-order to get
\begin{align}
\displaystyle \mathcal{O}(\eps^{-2}):&  \quad \mathrm{k}\partial_{\xi\xi}\widetilde{\theta}_0+\nu|\partial_\xi\widetilde{\mathbf{u}}_0|^2+(\nu+\lambda)|\partial_\xi\widetilde{\mathbf{u}}_0 \cdot \mathbf{n}|^2+\widetilde{\mu}_0^2=0,\label{theta-inner1}\\
\displaystyle \mathcal{O}(\eps^{-1}):&  \quad e_\theta(\widetilde{\rho}_0,\widetilde{\theta}_0 )(-\widetilde{\rho}_0\partial_\xi\widetilde{\theta}_0V_n+\widetilde{\rho}_0\widetilde{\mathbf{u}}_0 \cdot \mathbf{n}\partial_\xi\widetilde{\theta}_0) +\widetilde{\theta}_0 p_\theta(\widetilde{\rho}_0,\widetilde{\theta}_0 )\partial_\xi\widetilde{\mathbf{u}}_0 \cdot \mathbf{n}\nonumber \\
 &\quad =2\mathrm{k} (\mathbf{n}\cdot\nabla) \partial_\xi\widetilde{\theta}_0+\mathrm{k}\partial_\xi \widetilde{\theta}_0\kappa+\mathrm{k}\partial_{\xi\xi} \widetilde{\theta}_1+2 \nu\mathbb{D}\widetilde{\mathbf{u}}_0:\partial_\xi\widetilde{\mathbf{u}}_0 \otimes \mathbf{n} +2 \lambda \mathop{\mathrm{div}}\nolimits\widetilde{\mathbf{u}}_0 \partial_\xi\widetilde{\mathbf{u}}_0 \cdot \mathbf{n} \nonumber \\
  &\qquad +2\nu \partial_\xi\widetilde{\mathbf{u}}_0 \cdot \partial_\xi\widetilde{\mathbf{u}}_1+2(\nu+\lambda)\partial_\xi\widetilde{\mathbf{u}}_0 \cdot \mathbf{n} \partial_\xi\widetilde{\mathbf{u}}_1 \cdot \mathbf{n}+2\widetilde{\mu}_0\widetilde{\mu}_1. \label{theta-inner2}
\end{align}
For (C2) $M_\epsilon=\frac{1}{\epsilon}$, we have
\begin{align}
\displaystyle \mathcal{O}(\eps^{-3}):&  \quad \widetilde{\mu}_0=0,\label{theta-inner3}\\
\displaystyle \mathcal{O}(\eps^{-2}):&  \quad \mathrm{k}\partial_{\xi\xi}\widetilde{\theta}_0+\nu|\partial_\xi\widetilde{\mathbf{u}}_0|^2+(\nu+\lambda)|\partial_\xi\widetilde{\mathbf{u}}_0 \cdot \mathbf{n}|^2+2\widetilde{\mu}_0\widetilde{\mu}_1=0,\label{theta-inner4}\\
\displaystyle \mathcal{O}(\eps^{-1}):&  \quad e_\theta(\widetilde{\rho}_0,\widetilde{\theta}_0 )(-\widetilde{\rho}_0\partial_\xi\widetilde{\theta}_0V_n+\widetilde{\rho}_0\widetilde{\mathbf{u}}_0 \cdot \mathbf{n}\partial_\xi\widetilde{\theta}_0) +\widetilde{\theta}_0 p_\theta(\widetilde{\rho}_0,\widetilde{\theta}_0 )\partial_\xi\widetilde{\mathbf{u}}_0 \cdot \mathbf{n}\nonumber \\
 &\quad =2\mathrm{k} (\mathbf{n}\cdot\nabla) \partial_\xi\widetilde{\theta}_0+\mathrm{k}\partial_\xi \widetilde{\theta}_0\kappa+\mathrm{k}\partial_{\xi\xi} \widetilde{\theta}_1+2 \nu\mathbb{D}\widetilde{\mathbf{u}}_0:\partial_\xi\widetilde{\mathbf{u}}_0 \otimes \mathbf{n} +2 \lambda \mathop{\mathrm{div}}\nolimits\widetilde{\mathbf{u}}_0 \partial_\xi\widetilde{\mathbf{u}}_0 \cdot \mathbf{n} \nonumber \\
  &\qquad +2\nu \partial_\xi\widetilde{\mathbf{u}}_0 \cdot \partial_\xi\widetilde{\mathbf{u}}_1+2(\nu+\lambda)\partial_\xi\widetilde{\mathbf{u}}_0 \cdot \mathbf{n} \partial_\xi\widetilde{\mathbf{u}}_1 \cdot \mathbf{n}+\widetilde{\mu}_1^2+2\widetilde{\mu}_0\widetilde{\mu}_2. \label{theta-inner5}
\end{align}

Noting that the matching conditions for inner and outer expansions are as follows,
\begin{align}
&\lim_{\xi\rightarrow\pm\infty} \hat{g}_0(\mathbf{x},\xi)= g_0^{\pm}(\mathbf{x}), \label{match1}\\
&\lim_{\xi\rightarrow\pm\infty}  (\nabla_\mathbf{x} \hat{g}_0(\mathbf{x},\xi)+\partial_\xi\hat g_{1}(\mathbf{x},\xi)\mathbf{n})=\nabla g_0^{\pm}(\mathbf{x}).\label{match2}
\end{align}

%\subsubsection{Zeroth approximation} % (fold)
{\it Zeroth approximation.~~}
We consider the zeroth approximation in the inner expansion:
\begin{align}
&-\partial_\xi\widetilde{\rho}_0 V_\mathbf{n}+\partial_\xi(\widetilde{\rho}_0 \widetilde{\mathbf{u}}_0) \cdot \mathbf{n}=0,\label{rho-inner11}\\
&\partial_\xi\big((\nu+\lambda)\partial_\xi\widetilde{\mathbf{u}}_0\cdot \mathbf{n}\big)\mathbf{n}+ \partial_\xi\big(\nu\partial_\xi\widetilde{\mathbf{u}}_0\big)+\frac{1}{4}\partial_\xi ({\tilde{\phi}_0}^2-1)^2\mathbf{n}-\frac{1}{2}\partial_\xi|\partial_\xi \widetilde{\phi}_0|^2 \mathbf{n}=0,\label{uu-inner11}\\
&\chi_{M_\epsilon}\widetilde{\rho}_0 \partial_\xi \widetilde{\phi}_0 (V_\mathbf{n}- \widetilde{\mathbf{u}}_0\cdot \mathbf{n} )=\widetilde{\mu}_0,\label{phi-inner11}\\
&\widetilde{\rho}_0\widetilde{\mu}_0=-\partial_{\xi\xi}\widetilde{\phi}_0+{\tilde{\phi}_0}^3-\widetilde{\phi}_0,\label{mu-inner11}\\
&\mathrm{k}\partial_{\xi\xi}\widetilde{\theta}_0+\nu|\partial_\xi\widetilde{\mathbf{u}}_0|^2+(\nu+\lambda)|\partial_\xi\widetilde{\mathbf{u}}_0 \cdot \mathbf{n}|^2+\chi_{M_\epsilon}\widetilde{\mu}_0^2=0,\label{theta-inner11}
\end{align}
with boundary conditions:
\begin{align}\label{in-bc}
&\widetilde{\rho}_0(\mathbf{x},t,\pm \infty)=\rho_0^{\pm}(\mathbf{x},t),\\
&\widetilde{\mathbf{u}}_0(\mathbf{x},t,\pm \infty)=\mathbf{u}_0^{\pm}(\mathbf{x},t),\label{in-bc-uu}\\
&\widetilde{\phi}_0(\mathbf{x},t,\pm \infty)=\pm 1,\label{in-bc-phi}\\
%&\widetilde{\mu}_0(\mathbf{x},t,\pm \infty)=0 \label{in-bc-mu},\\
&\widetilde{\theta}_0(\mathbf{x},t,\pm \infty)=\theta_0^{\pm}(\mathbf{x},t) \label{in-bc-theta},
\end{align}
where $\chi_{M_\epsilon}$ is defined by \eqref{chi-eta}.

%From the equation \eqref{phi-inner11}, it follows that $\widetilde{\mu}_0$ is independent of $\xi$.

\begin{lemma}\label{lem-rho-u-phi}
Let $ (\widetilde{\rho}_0,\widetilde{\mathbf{u}}_0,\widetilde{\phi}_0,\widetilde{\theta}_0)$ be a solution of \eqref{rho-inner11}-\eqref{in-bc-phi}. It holds
\begin{align}
&\displaystyle \widetilde{\rho}_0(t,x,\xi)(V_n-\widetilde{\mathbf{u}}_0 \cdot \mathbf{n})(t,x,\xi) =\rho_0^{\pm}(\mathbf{x},t)(V_n-\mathbf{u}_0^{\pm}\cdot \mathbf{n})(\mathbf{x},t),\label{rho-u}\\
& \widetilde{\phi}_0(t,x,\xi)=\tanh(\frac{\xi}{\sqrt{2}}),\label{phi-xi}\\
&\partial_\xi \widetilde{\mathbf{u}}_0=0,\qquad \partial_\xi \widetilde{\theta}_0=0, \label{uuu}\\
&
\left\{\begin{array}{llll}
V_\mathbf{n}-\widetilde{\mathbf{u}}_0\cdot \mathbf{n}=0,&\ \mathrm{for~~(C1)}\ M_\epsilon=1, \\
\partial_\xi \widetilde{\rho}_0(V_\mathbf{n}-\widetilde{\mathbf{u}}_0\cdot \mathbf{n})=0,&\ \mathrm{for~~(C2)}\ M_\epsilon=\frac{1}{\epsilon}.
\end{array}\right.\label{vn-u}
\end{align}
\end{lemma}
\begin{proof}
From equation \eqref{rho-inner11} we have
\begin{align}\label{rho-u1}
 \displaystyle \partial_\xi(\widetilde{\rho}_0(V_{\mathbf{n}}-\widetilde{\mathbf{u}}_0 \cdot \mathbf{n})) =0,
 \end{align}
and by integration we obtain \eqref{rho-u}.
By multiplying \eqref{phi-inner11} with $ \partial_\xi \widetilde{\phi}_0$, and combining with \eqref{mu-inner11}, we obtain
\begin{align}\label{lem3-2}
\displaystyle -\chi_{M_\epsilon} (V_n-\widetilde{\mathbf{u}}_0\cdot \mathbf{n}) |\partial_\xi \widetilde{\phi}_0|^2 =\frac{1}{\widetilde{\rho}_0^2}\big(
 \frac{1}{2}\partial_\xi|\partial_\xi \widetilde{\phi}_0|^2-\frac14\partial_\xi (\tilde{\phi}_0^2-1)^2\big).
\end{align}
Multiplying equation \eqref{uu-inner11} by the norm vector $ \mathbf{n} $, we infer that
\begin{align}\label{lem2-1}
\displaystyle  \partial_\xi\big((\lambda+2\nu)\partial_\xi \widetilde{\mathbf{u}}_0\cdot \mathbf{n}\big)+\frac14\partial_\xi (\tilde{\phi}_0^2-1)^2-\frac{1}{2}\partial_\xi|\partial_\xi \widetilde{\phi}_0|^2=0.
\end{align}
By integrating equation \eqref{lem2-1} from $ -\infty $ to $ \xi $, we obtain
\begin{align}\label{uu-f-chi}
\displaystyle (\lambda+2\nu)\partial_\xi \widetilde{\mathbf{u}}_0\cdot \mathbf{n}+\frac14 (\tilde{\phi}_0^2-1)^2-\frac{1}{2}|\partial_\xi \widetilde{\phi}_0|^2=0.
\end{align}
Multiplying equation \eqref{uu-f-chi} by $\frac{\partial}{\partial \xi}(\frac{1}{\widetilde{\rho}_0^2})$, combining with \eqref{rho-inner11}, we infer that
\begin{align}\label{uu-phi}
\displaystyle -2(\lambda+2\nu)\widetilde{\rho}_0^{-4}(V_n-\widetilde{\mathbf{u}}_0\cdot \mathbf{n})|\partial_\xi \widetilde{\rho}_0|^2=\big(\frac{1}{2}|\partial_\xi \widetilde{\phi}_0|^2-\frac14 (\tilde{\phi}_0^2-1)^2\big)\frac{\partial}{\partial \xi}(\frac{1}{\widetilde{\rho}_0^2}).
\end{align}
Combining \eqref{lem3-2} with \eqref{uu-phi}, and integrating them from $-\infty$ to $+\infty$, we obtain
\begin{align}\label{uu-phi-integ}
\displaystyle \int_{-\infty}^{+\infty}(V_n-\widetilde{\mathbf{u}}_0\cdot \mathbf{n})\big(\chi_{M_\epsilon} |\partial_\xi \widetilde{\phi}_0|^2+2(\lambda+2\nu)\widetilde{\rho}_0^{-4}|\partial_\xi \widetilde{\rho}_0|^2\big)d\xi=0.
\end{align}
Since $\widetilde{\rho}_0>0$, $\lambda+2\nu>0$, and \eqref{chi-eta}, combining \eqref{uu-phi-integ} and \eqref{rho-u} we obtain \eqref{vn-u}.
Then from \eqref{phi-inner11} and \eqref{lem3-2}, for both (C1) and (C2), we obtain
\begin{align}\label{f-chi}
\displaystyle \widetilde{\mu}_0=0,\quad
\frac{1}{2}|\partial_\xi \widetilde{\phi}_0|^2= \frac14(\widetilde{\phi}_0^2-1)^2.
\end{align}
Then we have
\begin{align}\label{pxi-chi}
\displaystyle  \partial_\xi \widetilde{\phi}_0=\frac{1}{\sqrt{2}}(1- \widetilde{\phi}_0^2).
\end{align}
Noticing $\widetilde{\phi}_0(\mathbf{x},t,0)=0 $, we obtain \eqref{phi-xi}.
Since $ \lambda+2\nu>0 $, combining \eqref{uu-f-chi} and \eqref{f-chi}, we have
\begin{align}\label{uu-norm}
\displaystyle \partial_\xi \widetilde{\mathbf{u}}_0\cdot \mathbf{n}=0.
\end{align}
Multiplying equation \eqref{uu-inner2} by the tangential vector $ \mathbf{\tau} $, we infer that
\begin{align}
\displaystyle  \partial_\xi(\lambda\partial_\xi\widetilde{\mathbf{u}}_0)\cdot \mathbf{\tau}=0.
\end{align}
Since $ \lambda>0 $, by integration from $ -\infty $ to $ \xi $,  we have
\begin{align}\label{u-tau}
\displaystyle  \partial_\xi\widetilde{\mathbf{u}}_0\cdot \mathbf{\tau}=0.
\end{align}
By integrating equation \eqref{theta-inner11} from $ -\infty $ to $ \xi $,   combining with \eqref{uu-norm} and \eqref{f-chi}, we obtain
\begin{align}
\displaystyle \mathrm{k} \partial_\xi \widetilde{\theta}_0=0.
\end{align}
Since $ \mathrm{k}>0 $, combining \eqref{uu-norm} and \eqref{u-tau}, we derive \eqref{uuu}. The  proof is completed.
\end{proof}

\begin{remark}\label{rem:s1s2} For $(C1)~ M_\epsilon=1$, we have $V_{n}-\mathbf{u}_0 \cdot \mathbf{n}=0$ on $\Gamma$. For $(C2)~ M_\epsilon=\frac{1}{\epsilon}$ we introduce the set $S \subset \Gamma$  by
\begin{align}
\displaystyle \rho^+(\mathbf{x},t)=\rho^-(\mathbf{x},t), & \quad \text{on}~ S \times(-\infty,+\infty).
\end{align}
From \eqref{vn-u}, it implies that
\begin{align}
\displaystyle \rho^+(\mathbf{x},t)=\rho^-(\mathbf{x},t),\quad V_n-\mathbf{u}_0 \cdot \mathbf{n} = 0, & \quad \text{on}~ \Gamma\setminus S \times(-\infty,+\infty).
\end{align}
\end{remark}

{\it First approximation.~~}
We derive the system from the first order asymptotic analysis by the equations \eqref{rho-inner2}, \eqref{uu-inner2}, \eqref{mu-inner2} and \eqref{theta-inner2}. By \eqref{uuu}, the equations are simplified to the following system:
\begin{align}
&\partial_t \widetilde{\rho}_0+\textrm{div}(\widetilde{\rho}_0 \widetilde{\mathbf{u}}_0)-
\partial_\xi\widetilde{\rho}_1 (V_n-\widetilde{\mathbf{u}}_0\cdot \mathbf{n})+ \partial_{\xi}(\widetilde{\rho}_0 \widetilde{\mathbf{u}}_1 \cdot \mathbf{n})=0, \label{rho-inner22}\\
&-\partial_\xi p_0 \mathbf{n}+\frac14\nabla (\tilde{\phi}_0^2-1)^2+\partial_\xi ((\tilde{\phi}_0^3-\widetilde{\phi}_0)\widetilde{\phi}_1)\mathbf{n}+\partial_\xi((\nu+\lambda)\partial_\xi\hat{\mathbf{u}}_1\cdot \mathbf{n})\mathbf{n}+ \partial_\xi(\nu\partial_\xi\hat{\mathbf{u}}_1)\nonumber\\
&~~~\qquad=\partial_\xi\big(\partial_\xi \widetilde{\phi}_0\partial_\xi \widetilde{\phi}_1\big)\mathbf{n}+\mathbf{n}\cdot\nabla|\partial_\xi \widetilde{\phi}_0|^2 \mathbf{n}+|\partial_\xi \widetilde{\phi}_0|^2 \kappa \mathbf{n}
+\partial_{\xi\xi}\widetilde{\phi}_0 \nabla\widetilde{\phi}_0, \label{uu-inner22}\\
&\widetilde{\rho}_0\widetilde{\mu}_1=-2(\mathbf{n}\cdot\nabla) \partial_\xi\widetilde{\phi}_0-\partial_\xi \widetilde{\phi}_0\kappa-\partial_{\xi\xi} \widetilde{\phi}_1+(3{\tilde{\phi}_0}^2-1)\widetilde{\phi}_1,\label{mu-inner22}\\
&\mathrm{k}\partial_{\xi\xi} \widetilde{\theta}_1=0,\label{theta-inner22}
\end{align}
and for (C2) $M_\epsilon=\frac{1}{\epsilon}$ we have
\begin{align}
\widetilde{\rho}_0\partial_\xi \widetilde{\phi}_0 (V_n-\widetilde{\mathbf{u}}_0 \cdot \mathbf{n})
=\widetilde{\mu}_1. \label{phi-inner22}
\end{align}

\begin{lemma}\label{lem-tensor}
%Let $ (\widetilde{\rho}_0,\widetilde{\mathbf{u}}_0,\widetilde{\phi}_0)$ be a solution of \eqref{rho-inner11}-\eqref{in-bc-phi}.
%Letting $\epsilon \rightarrow 0 $ in ansatz \eqref{inner-e},
It holds
\begin{align}\label{tensor-n}
\displaystyle [(\mathbb{S}_0-p_0\mathbb{I})\mathbf{n}]_\Gamma=\sigma \kappa \mathbf{n}, \quad \text{on}~ \Gamma,
\end{align}
where $ \mathbb{S}_0^{\pm}=2\nu\mathbb{D}\mathbf{u}_0^{\pm}+\lambda\mathrm{div}\mathbf{u}_0^{\pm}\mathbb{I}$, $p_0^\pm=(\rho_0^{\pm})^2e_\rho(\rho_0^{\pm},\theta_0^{\pm})+\theta_0^{\pm}p_\rho(\rho_0^{\pm},\theta_0^{\pm})$.
\end{lemma}
\begin{proof}
Since $ \widetilde{\phi}_0 $ is given by \eqref{phi-xi}, then we simplify equation \eqref{uu-inner22} to
\begin{align}
&-\partial_\xi p_0 \mathbf{n}+\partial_\xi \big((\lambda+\nu)\partial_{\xi}\widetilde{\mathbf{u}}_1+\nu (\partial_\xi\widetilde{\mathbf{u}}_1\cdot \mathbf{n})\mathbf{n}\big)\nonumber\\
&~~~=\partial_\xi\big(\partial_\xi \widetilde{\phi}_0\partial_\xi \widetilde{\phi}_1\big)\mathbf{n}+|\partial_\xi \widetilde{\phi}_0|^2 \kappa \mathbf{n}
-\partial_\xi ((\tilde{\phi}_0^3-\widetilde{\phi}_0)\widetilde{\phi}_1)\mathbf{n}. \label{uu-inner3}
\end{align}
We integrate \eqref{uu-inner3} from $ \xi=-\infty $ to $ \xi=+\infty $, and the first and third summand on the right-hand-side vanish. The first summand on the left-hand-side of \eqref{uu-inner3} turns into
\begin{align}
\displaystyle  -\int_{-\infty}^{\infty}\partial_\xi p_0 \mathbf{n}d \xi=-(p(\rho_0^+,\theta_0^+)-p(\rho_0^-,\theta_0^-))\mathbf{n}=-[p(\rho_0,\theta_0)]\mathbf{n}.
\end{align}
%where $p_0^\pm=p(\rho_0^\pm)$.
From the matching condition \eqref{match1}-\eqref{match2}, we obtain
\begin{align}\label{newmatch}
\displaystyle \lim_{\xi \rightarrow \pm \infty} \partial_\xi \mathbf{u}_1=\mathbf{n}\cdot \nabla \mathbf{u}_0.
\end{align}
For the second summand on the left-hand-side of \eqref{uu-inner3} we have
\begin{align}
&\int_{-\infty}^{\infty}\partial_\xi\big(((\nu+\lambda)\partial_\xi\widetilde{\mathbf{u}}_1\cdot \mathbf{n})\mathbf{n}+ \nu\partial_\xi\widetilde{\mathbf{u}}_1\big)d \xi \nonumber\\
&= \big(((\nu+\lambda)\partial_\xi\widetilde{\mathbf{u}}_1\cdot \mathbf{n})\mathbf{n}+ \nu\partial_\xi\widetilde{\mathbf{u}}_1\big)|^{\xi=+\infty}_{\xi=-\infty}\nonumber\\
&=[\lambda\mathop{\mathrm{div}}\nolimits \widetilde{\mathbf{u}}_0 \mathbf{n} +\nu(\nabla \widetilde{\mathbf{u}}_0+ \nabla ^{\top} \widetilde{\mathbf{u}}_0)\mathbf{n}]\nonumber\\
&=[\mathbb{S}_0]\mathbf{n}.\label{uu-inner33}
\end{align}
Combining the equations above together we conclude
\begin{align}
\displaystyle -[p(\rho_0,\theta_0)\mathbf{n}]+[\mathbb{S}_0\mathbf{n}]=\int_{-\infty}^{\infty}
|\partial_\xi \widetilde{\phi}_0|^2 \kappa \mathbf{n}d \xi=\sigma \kappa \mathbf{n},
\end{align}
where
\begin{align}
 \displaystyle  \sigma= \int_{-\infty}^{\infty}
|\partial_\xi \widetilde{\phi}_0|^2d \xi=\frac{2}{3}\sqrt{2}, \label{sigma0}
 \end{align}
then it follows \eqref{tensor-n}.
\end{proof}

For $(C1)~M_\epsilon=1$, from Lemma \ref{lem-rho-u-phi} and Lemma \ref{lem-tensor}, we obtain the following jump conditions on the free interface $\Gamma$.
\begin{lemma}\label{c1-jump}
%Let $\widetilde{\rho}_0$ be an arbitrary, positive and continuous function with boundary condition \eqref{in-bc}.
Letting $\epsilon \rightarrow 0$ in ansatz \eqref{inner-e} for $(C1)~M_\epsilon=1$, it holds
\begin{align}
 \displaystyle &[\mathbf{u}]_\Gamma=0,\quad [\theta]_\Gamma=0,& \text{on}~~\Gamma,\label{model-jump-u}\\
 \displaystyle &V_\mathbf{n}-\mathbf{u} \cdot \mathbf{n} =0,& \text{on}~~\Gamma,\label{RH}
 \\
 \displaystyle &\big[\big(2\nu\mathbb{D}\mathbf{u}+\lambda\mathrm{div}\mathbf{u}\mathbb{I}-\big( \rho^2\frac{\partial e}{\partial \rho}+\theta\frac{\partial p}{\partial \theta}\big) \mathbb{I}\big)\mathbf{n}\big]_\Gamma=\sigma \kappa \mathbf{n},& \text{on}~~\Gamma.
 %\\ V_n=\mathbf{u}\cdot \mathbf{n},& \text{on}~~\Gamma,
 \label{model-jump-t}
\end{align}
\end{lemma}

For $(C2)~M_\epsilon=\frac{1}{\epsilon}$, we derive the equation \eqref{phi-inner22} to obtain the following lemma.
\begin{lemma}\label{lem-c1c2}
%Let $ (\widetilde{\rho}_0,\widetilde{\mathbf{u}}_0,\widetilde{\phi}_0)$ be a solution of \eqref{rho-inner11}-\eqref{in-bc-phi}.
%Letting $\epsilon \rightarrow 0$ in ansatz \eqref{inner-e},
It holds
\begin{align}
\widetilde{\rho}_0(V_\mathbf{n}-{\mathbf{u}}_0 \cdot \mathbf{n})\int_{-\infty}^{+\infty} \widetilde{\rho}_0 |\partial_\xi \widetilde{\phi}_0|^2 d \xi+   \sigma \kappa=0. \label{vv-rho-s}
\end{align}
\end{lemma}
\begin{proof}
Combining \eqref{mu-inner22} and \eqref{phi-inner22}, we have
\begin{align}
\displaystyle  -\widetilde{\rho}_0^2(V_\mathbf{n}-{\mathbf{u}}_0 \cdot \mathbf{n}) \partial_\xi \widetilde{\phi}_0
=2(\mathbf{n}\cdot\nabla) \partial_\xi\widetilde{\phi}_0+\partial_\xi \widetilde{\phi}_0\kappa+\partial_{\xi\xi} \widetilde{\phi}_1-(3\tilde{\phi}_0^2-1) \widetilde{\phi}_1,\label{1app-phi0}
\end{align}
multiplying \eqref{1app-phi0}  by $ \partial_\xi \widetilde{\phi}_0 $ and integrating from $ \xi=-\infty $ to $ \xi=+\infty $ , we obtain
\begin{eqnarray}
&&-\int_{-\infty}^{+\infty}\widetilde{\rho}_0^2(V_n-{\mathbf{u}}_0 \cdot \mathbf{n}) |\partial_\xi \widetilde{\phi}_0|^2 d \xi\notag\\
&&=\int_{-\infty}^{+\infty}  |\partial_\xi \widetilde{\phi}_0|^2\kappa+\partial_{\xi\xi} \widetilde{\phi}_1\partial_\xi \widetilde{\phi}_0-(3\tilde{\phi}_0^2-1) \widetilde{\phi}_1\partial_\xi \widetilde{\phi}_0 d\xi.
\label{1app-phi}
\end{eqnarray}
From \eqref{mu-inner1}, \eqref{sigma0} and \eqref{rho-u}, we have
\begin{align}
 &-\widetilde{\rho}_0(V_\mathbf{n}-{\mathbf{u}}_0 \cdot \mathbf{n})\int_{-\infty}^{+\infty} \widetilde{\rho}_0 |\partial_\xi \widetilde{\phi}_0|^2 d \xi \nonumber\\
  =&  \int_{-\infty}^{+\infty}|\partial_\xi \widetilde{\phi}_0|^2d \xi\kappa+  \int_{-\infty}^{+\infty}\big(\partial_{\xi\xi}\widetilde{\phi}_1\partial_\xi \widetilde{\phi}_0-(3\tilde{\phi}_0^2-1)\partial_\xi \widetilde{\phi}_0 \widetilde{\phi}_1\big) d \xi\nonumber\\
 =&  \sigma \kappa-  \int_{-\infty}^{+\infty}\partial_{\xi}\widetilde{\phi}_1\big(-\partial_{\xi\xi} \widetilde{\phi}_0+(\tilde{\phi}_0^3-\widetilde{\phi}_0)\big) d \xi =\sigma \kappa,  \label{phi-inner33}
\end{align}
that is \eqref{vv-rho-s}.
\end{proof}
For $(C2)~M_\epsilon=\frac{1}{\epsilon}$, from Lemma \ref{lem-rho-u-phi}, Lemma \ref{lem-tensor}, Lemma \ref{lem-c1c2} and Remark \ref{rem:s1s2}, we obtain the following jump conditions on the free interface $\Gamma$.
\begin{lemma}\label{c2-jump}
Letting $\epsilon \rightarrow 0$ in ansatz \eqref{inner-e} for $(C2)~M_\epsilon=\frac{1}{\epsilon}$, it holds
\begin{equation}\label{c2-jump1}
 \left\{\begin{array}{llll}
 \displaystyle [\mathbf{u}]_\Gamma=0,\quad [\theta]_\Gamma=0,\quad [\rho]_\Gamma=0,\\
 \displaystyle \big[\big(2\nu\mathbb{D}\mathbf{u}+\lambda\mathrm{div}\mathbf{u}\mathbb{I}-\big( \rho^2\frac{\partial e}{\partial \rho}+\theta\frac{\partial p}{\partial \theta}\big)\mathbb{I}\big)\mathbf{n}\big]_\Gamma=\sigma \kappa \mathbf{n},\\
 \displaystyle \rho^2(V_\mathbf{n}-\mathbf{u} \cdot \mathbf{n}) =- \kappa,
\end{array}\right.  \text{on}~~S,
\end{equation}
\end{lemma}
\begin{equation}\label{c2-jump2}
 \left\{\begin{array}{llll}
 \displaystyle [\mathbf{u}]_\Gamma=0,\quad [\theta]_\Gamma=0,\\
 \displaystyle V_\mathbf{n}-\mathbf{u} \cdot \mathbf{n} =0,
 \\
 \displaystyle \big[\big(2\nu\mathbb{D}\mathbf{u}+\lambda\mathrm{div}\mathbf{u}\mathbb{I}-\big( \rho^2\frac{\partial e}{\partial \rho}+\theta\frac{\partial p}{\partial \theta}\big)\mathbb{I}\big)\mathbf{n}\big]_\Gamma=0,
\end{array}\right. \text{on}~~\Gamma \setminus S.
\end{equation}

{\it Proof of Theorem \ref{thm-main}.~~}
We analyze the limit process by the matched asymptotic expansion. By applying the results of Lemma \ref{lem-outer}, Lemma \ref{c1-jump} and Lemma \ref{c2-jump}, the system \eqref{NSAC} converges to the sharp interface problem \eqref{free boundary problem for NSAC}  with the free interface \eqref{jump condition-case-1}-\eqref{jump condition-case-2}.

\end{document}